\documentclass{birkjour}
 \newtheorem{thm}{Theorem}[section]
 
 \newtheorem{lem}[thm]{Lemma}
 
 \theoremstyle{definition}
 
 \theoremstyle{remark}

 \numberwithin{equation}{section}
\usepackage[english]{babel}

\def\diag{{\rm diag}}

\def\R{{\mathbb R}}
\def\C{{\mathbb C}}

\def\Spin{{\rm Spin}}
\def\Pin{{\rm Pin}}

\def\cl{{C}\!\ell}
\def\Cl{{\rm C}}
\def\ad{{\rm ad}}
\def\Aut{{\rm Aut}}

\def\G{{\rm G}}

\def\P{{\rm P}}
\def\A{{\rm A}}
\def\B{{\rm B}}
\def\Q{{\rm Q}}
\def\Z{{\rm Z}}
\def\mod{{\,\rm mod}}
\def\la{\langle}
\def\ra{\rangle}

  \usepackage{tikz}

\begin{document}

%
%
%
%
%
%
%
%
%

\title[On inner automorphisms preserving subspaces of Clifford algebras]
 {On inner automorphisms preserving fixed\\ subspaces of Clifford algebras}

\author[Dmitry Shirokov]{Dmitry Shirokov}

\address{%
HSE University\\
Myasnitskaya str. 20\\
101000 Moscow\\
Russia}
\address{
Institute for Information Transmission Problems of Russian Academy of Sciences\\
Bolshoy Karetny per. 19\\
127051 Moscow\\
Russia}
\email{dm.shirokov@gmail.com}



\date{May, 2020}

\begin{abstract}
In this paper, we consider inner automorphisms that leave invariant fixed subspaces of real and complex Clifford algebras -- subspaces of fixed grades and subspaces determined by the reversion and the grade involution. We present groups of elements that define such inner automorphisms and study their properties. Some of these Lie groups can be interpreted as generalizations of Clifford, Lipschitz, and spin groups. We study the corresponding Lie algebras. Some of the results can be reformulated for the case of more general algebras -- graded central simple algebras or graded central simple algebras with involution.
\end{abstract}

\maketitle

\section{Introduction}\label{sectIntr}

In this paper, we study inner automorphisms that leave invariant fixed fundamental subspaces of the real and complex Clifford algebras.  We present groups of elements that define such inner automorphisms.  Some of these groups are associated with the vector subspace $V$ of the Clifford algebra $\cl(V,Q)$ \cite{Chev, Port} over a quadratic space $(V, Q)$ with a nondegenerate quadratic form $Q$: well-known Clifford group $\Gamma$, which preserves the subspace $V$ and is widely used in the theory of spin groups; the groups $\Gamma^k$, $k=0, 1, \ldots, n$, which preserve the subspaces of fixed grades $k$. The others are naturally defined by the main involution (the grade involution) and the reversion. We introduce the groups $\P$ (preserve subspaces of fixed parity), $\A$, $\B$, $\Q$, and $\Q^\prime$ (preserve some other fundamental subspaces determined by the reversion and the grade involution) and study their properties. As one of the anonymous reviewers noted, these subspaces can be defined in more general algebras, and some of the results of this paper can be reformulated for the more general case. Namely, the groups $\P$ (see Section \ref{sectP}) can be introduced in the graded central simple algebras (GCSAs) introduced by  Wall in 1964 \cite{Wall}. The groups $\A$, $\B$, $\Q$ (see Sections \ref{sectA}, \ref{sectB}, \ref{sectQ}) can be introduced in the graded central simple algebras with involution (GCSAsWI) introduced by  Wall in 1968 \cite{Wall2}\footnote{The area of interest of the author of this paper is mostly the applications of some specific algebras (as Clifford algebras) in physics, in particular, in the field theory. Therefore, let us present results only for the particular case of Clifford algebras in this paper. One can reformulate the statements for the more general cases of GCSAs or GCSAsWI, if the need arises for some purpose. We accompanied the statements of this paper with footnotes, including those received from the reviewer, about their possible reformulation for the  more general case.}. Let us also note the paper \cite{Helm} on GSCAs and GSCAsWI, including, as a special case, the Clifford algebras. Clifford algebras are widely used in different applications -- in engineering, physics, robotics, computer vision, image and signal processing, etc. Groups that preserve different structures of Clifford algebras under similarity transformation may be of interest to these applications. All these groups contain spin groups as subgroups and can be considered as generalizations of Clifford, Liptshitz, and spin groups.



Let $\Cl$ be the real Clifford algebra $\cl(V,Q)$, $V=\R^{p,q}$, $p+q=n$ (or the geometric algebra \cite{Hestenes, Lasenby}) or the complex Clifford algebra $\cl(\C^n)$ \cite{Lounesto, Port, LM, BT}. We denote the identity element by $e$ and the generators of $\Cl$ by $e_a$, $a=1, \ldots, n$. In the case of the real Clifford algebra, the generators $e_a$, $a=1, \ldots, n$ satisfy
$$e_a e_b+e_b e_a=2\eta_{ab}e,\qquad a, b=1, \ldots, n,$$
where $\eta=(\eta_{ab})=\diag(1,\ldots, 1, -1, \ldots, -1)$ is the diagonal matrix with its first $p$ entries equal to $1$ and the last $q$ entries equal to $-1$ on the diagonal. In the case of the complex Clifford algebra, the generators satisfy the same conditions but with the identity matrix $\eta=I_n$ of size $n$. Consider the subspaces $\Cl^k$ of grades $k=0, 1, \ldots, n$. The elements of these subspaces are linear combinations of the basis elements $e_{a_1 \ldots a_k}:=e_{a_1}\cdots e_{a_k}$, $a_1 < a_2 < \cdots < a_k$, with multi-indices of length $k$. The Clifford algebra $\Cl$ can be represented as the direct sum of the even and odd subspaces
\begin{eqnarray*}
&&\Cl=\Cl^{(0)}\oplus \Cl^{(1)},\qquad \Cl^{(0)}=\!\!\bigoplus_{k=0 \mod 2}\!\! \Cl^k,\qquad \Cl^{(1)}=\!\!\bigoplus_{k=1 \mod 2}\!\! \Cl^k,\\
&&\Cl^{(j)}=\{U\in \Cl:\, \hat{U}=(-1)^j U\},\quad j=0, 1,
\end{eqnarray*}
where $\hat{\quad}$ is the grade involution (or the main involution). Also we consider the following four subspaces (see \cite{quat, quat1, quat2}), which are naturally determined by the grade involution $\hat{\quad}$ and the reversion $\tilde{\quad}$\,:
\begin{eqnarray}
&&\Cl^{\overline{m}}:=\!\!\bigoplus_{k=m \mod 4}\!\! \Cl^k,\qquad m=0, 1, 2, 3,\nonumber\\
&&\Cl^{\overline{m}}=\{U\in \Cl:\,  \hat{U}=(-1)^m U,\, \tilde{U}=(-1)^{\frac{m(m-1)}{2}} U\}.\label{quat}
\end{eqnarray}
The algebra $\Cl$ is a $Z_2\times Z_2$-graded algebra w.r.t. these four subspaces and the operations of commutator $[U,V]=UV-VU$ and anticommutator $\{U, V\}=UV+VU$:
\begin{eqnarray}
&&[\Cl^{\overline{m}}, \Cl^{\overline{m}}]\subset \Cl^{\overline{2}},\quad [\Cl^{\overline{m}}, \Cl^{\overline{2}}]\subset \Cl^{\overline{m}},\quad m=0, 1, 2, 3,\nonumber\\
&&[\Cl^{\overline{0}}, \Cl^{\overline{1}}]\subset \Cl^{\overline{3}},\quad [\Cl^{\overline{0}}, \Cl^{\overline{3}}]\subset \Cl^{\overline{1}},\quad [\Cl^{\overline{1}}, \Cl^{\overline{3}}]\subset \Cl^{\overline{0}},\nonumber\\
&&\{\Cl^{\overline{m}}, \Cl^{\overline{m}}\}\subset \Cl^{\overline{0}},\quad \{\Cl^{\overline{m}}, \Cl^{\overline{0}}\}\subset \Cl^{\overline{m}},\quad m=0, 1, 2, 3,\nonumber\\
&&\{\Cl^{\overline{1}}, \Cl^{\overline{2}}\}\subset \Cl^{\overline{3}},\quad \{\Cl^{\overline{2}}, \Cl^{\overline{3}}\}\subset \Cl^{\overline{1}},\quad \{\Cl^{\overline{3}}, \Cl^{\overline{1}}\}\subset \Cl^{\overline{2}}.\nonumber
\end{eqnarray}
We use the notation $A^\times$ for the set of invertible elements of any set $A$. Consider the adjoint representation (inner automorphism) $\ad: \Cl^{\times}\rightarrow \Aut \Cl$ acting on the group of all invertible elements of the Clifford algebra
$$\Cl^{\times}=\{T\in \Cl: \,\exists \, T^{-1}\}$$
as $T\mapsto \ad_T$, where $ \ad_T U=TUT^{-1}$ for any $U\in \Cl$. The kernel of $\ad$ is
$$\ker(\ad)=\{T\in \Cl^\times:\quad \ad_T(U)=U \quad \forall U\in \Cl\}=\Z^\times,$$
where $\Z^\times$ is the group of all invertible elements of the center of $\Cl$
\begin{equation}\label{center}
\Z:=\left\lbrace
\begin{array}{ll}
\Cl^0, & \mbox{if $n$ is even,}\\
\Cl^0\oplus \Cl^n, & \mbox{if $n$ is odd.}
\end{array}
\right.
\end{equation}
Let us consider the well-known Clifford group \cite{BT, Chev, Port, Bulg}, which consists of elements that define inner automorphisms preserving the subspace of grade~$1$:
\begin{eqnarray}
\Gamma&:=&\{T\in \Cl^\times:\quad T \Cl^{1}T^{-1}\subseteq \Cl^{1}\}\label{2.2}\\
&=&\{ W v_1 \cdots v_m:\quad m\leq n,\quad W\in \Z^\times,\quad v_j\in \Cl^{\times 1}\},\label{2.3}
\end{eqnarray}
where $\Cl^{\times 1}=\{v_j\in \Cl^1: v_j^2\neq 0\}$ is the subset of all invertible elements of grade $1$. The equivalence of (\ref{2.2}) and (\ref{2.3}) is well-known and can be proved using the Cartan-Dieudonn\'e theorem \cite{Diedonne, Bulg} or the generalization of the Pauli's theorem \cite{DAN, Paulispin}. From (\ref{2.2}) and (\ref{2.3}), we can easily obtain
\begin{eqnarray}
\Gamma=\{T\in \Z^\times(\Cl^{\times (0)}\cup \Cl^{\times (1)}):\quad T \Cl^{1} T^{-1}\subseteq \Cl^{1}\}.\label{5.23}
\end{eqnarray}
In this paper, we present and study groups that leave invariant different other fixed subspaces (the subspaces of fixed grades, the subspaces defined by the reversion and the grade involution) of the Clifford algebra $\Cl$. In the case $n=1$, the Clifford algebra $\Cl$ is commutative, and all such groups coincide with $\Cl^\times$. We consider the case $n\geq 2$ below. We summarize the results for the Lie groups and Lie algebras presented in this paper in Table~\ref{table1}.

We use notations with multi-indices for direct sums of different subspaces:
$$\Cl^{(k)\overline{lm}r}:=\Cl^{(k)}\oplus \Cl^{\overline{l}} \oplus \Cl^{\overline{m}}\oplus \Cl^r$$
and similar ones. The group of elements that define inner automorphisms preserving the subspace $\Cl^{(k)\overline{lm}r}$ is denoted by $\Gamma^{(k)\overline{lm}r}$.

\begin{table}[ht]
\caption{Lie groups preserving fixed subspaces of $\Cl$ under similarity transformation and corresponding Lie algebras}
\begin{tabular}{cccc} \hline
Lie group & $n$ & Lie algebra & dimension \\  \hline
 $\Cl^\times$& & $\Cl$ & $2^n$ \\ \hline
$\Gamma=\bigcap_{k=0}^n \Gamma^k$ & $1\mod 2$ & $\Cl^{02n}$ & $\frac{n(n-1)}{2}+2$ \\
 & $0\mod 2$ &$\Cl^{02}$ & $\frac{n(n-1)}{2}+1$\\ \hline
$\P=\Gamma^{(0)}=\Gamma^{(1)}$ & $1\mod 2$ &$\Cl^{(0)n}$ & $2^{n-1}+1$ \\
& $0\mod 2$ &$\Cl^{(0)}$ & $2^{n-1}$ \\ \hline
$\A=\Gamma^{\overline{01}}=\Gamma^{\overline{23}}$ & $1\mod 4$ & $\Cl^{0\overline{23}n}$  & $2^{n-1}-2^{\frac{n-1}{2}}\sin (\frac{\pi(n+1)}{4})+2$ \\
& $0, 2, 3\mod 4$ & $\Cl^{0\overline{23}}$ & $2^{n-1}-2^{\frac{n-1}{2}}\sin (\frac{\pi(n+1)}{4})+1$ \\ \hline
$\B=\Gamma^{\overline{03}}=\Gamma^{\overline{12}}$ & $3\mod 4$ & $\Cl^{0\overline{12}n}$ & $2^{n-1}-2^{\frac{n-1}{2}}\cos (\frac{\pi(n+1)}{4})+2$ \\
 &$0, 1, 2\mod 4$ & $\Cl^{0\overline{12}}$ & $2^{n-1}-2^{\frac{n-1}{2}}\cos (\frac{\pi(n+1)}{4})+1$ \\ \hline
$\Q=\Q^\prime=\Gamma^{\overline{k}}$ & $1, 3\mod 4$ & $\Cl^{0\overline{2}n}$ & $2^{n-2}-2^{\frac{n-2}{2}}\cos (\frac{\pi n}{4})+2$ \\
$(k=0, 1, 2, 3)$ & $2\mod 4$  & $\Cl^{0\overline{2}}$ & $2^{n-2}-2^{\frac{n-2}{2}}\cos (\frac{\pi n}{4})+1$ \\ \hline
$\Q=\Gamma^{\overline{1}}=\Gamma^{\overline{3}}$ & $0\mod 4$  & $\Cl^{0\overline{2}}$ & $2^{n-2}-2^{\frac{n-2}{2}}\cos (\frac{\pi n}{4})+1$ \\ \hline
$\Q^\prime=\Gamma^{\overline{0}}=\Gamma^{\overline{2}}$ & $0\mod 4$ & $\Cl^{0\overline{2}n}$ & $2^{n-2}-2^{\frac{n-2}{2}}\cos (\frac{\pi n}{4})+2$ \\ \hline
\end{tabular}
\label{table1}
\end{table}

%

We can illustrate the considered Lie groups in the following way (see Figure \ref{figure1}). Note that we have $\Gamma\subseteq \Q\subseteq\Q^\prime\subseteq\P$ and $\Q=\A\cap\B=\A\cap\P=\B\cap\P$. All the considered groups are subgroups of the group $\Cl^\times$. Spin groups are subgroups of the smallest of the considered groups $\Gamma$.

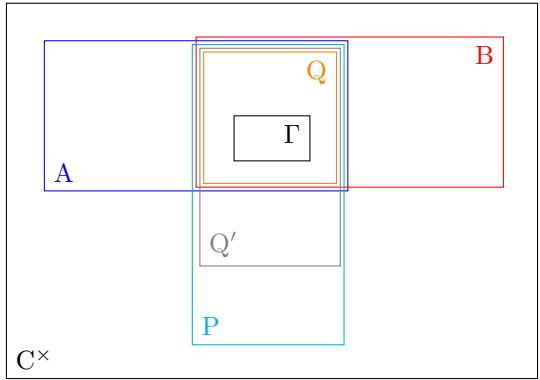
\begin{figure}[ht]
\centering
    \begin{tikzpicture}
    \draw[blue] (4, 2 ) rectangle (0,0) node[above right] {$\A$};
    \draw[red] (2.0,0.05) rectangle (6.05,2.05) node[below left] {$\B$};
    \draw[cyan] (3.95, 1.95) rectangle(1.95,-2.05) node[ above right]{$\P$};
    \draw[black] (2.5, 0.4) rectangle (3.5, 1.0) node[below left]{$\Gamma$};
    \draw[orange] (2.1, 0.1) rectangle (3.85, 1.85) node[below left]{$\Q$};
    \draw[black] (6.5, 2.5) rectangle (-0.5, -2.5) node[above right]{$\Cl^\times$};
    \draw[gray] (3.9, 1.9) rectangle (2.05, -1.0) node[above right]{$\Q^\prime$};
    \end{tikzpicture}
\caption{Lie groups preserving fixed subspaces of $\Cl$ under similarity transformation} \label{figure1}
\end{figure}

\section{The Lie groups $\Gamma^k$}\label{sectGamma}

Let us use the following notation for the groups that preserve subspaces of fixed grades under similarity transformation
\begin{eqnarray}
\Gamma^k:=\{T\in \Cl^\times:\quad T \Cl^{k}T^{-1}\subseteq \Cl^{k}\},\quad k=0, 1, \ldots, n.\label{gammak}
\end{eqnarray}
In the particular case, we get the Clifford group $\Gamma^1:=\Gamma$ (see (\ref{2.2})).

\begin{lem}\label{lemmagammak}
We have
$$\Gamma^0=\Cl^\times,\qquad
\Gamma^n=\left\lbrace
\begin{array}{ll}
\Cl^\times, & \mbox{if $n$ is odd,}\\
\Cl^{\times (0)}\cup \Cl^{\times (1)}, & \mbox{if $n$ is even.}
\end{array}\nonumber
\right.
$$
\end{lem}
\begin{proof} We have $\Gamma^0=\Cl^\times$ in the case of arbitrary $n$ and $\Gamma^n=\Cl^\times$ in the case of odd $n$ because of the center (\ref{center}) of $\Cl$.

In the case of even $n$, we can easily verify that $\Cl^{\times (0)}\cup \Cl^{\times (1)}\subseteq \Gamma^n$ because $e_{1\ldots n}$ commutes with all even elements and anticommutes with all odd elements. Let us prove that $\Gamma^n\subseteq \Cl^{\times (0)}\cup \Cl^{\times (1)}$. Suppose that $T=T_0+T_1\in\Gamma^n$, where $T_0\in \Cl^{(0)}$ and $T_1\in \Cl^{(1)}$. We get
$$
(T_0+T_1)e_{1\ldots n}(T_0+T_1)^{-1}=e_{1\ldots n}(T_0-T_1)(T_0+T_1)^{-1}=\lambda e_{1\ldots n}
$$
for some constant $\lambda$, which is equivalent to $(T_0-T_1)=\lambda (T_0+T_1)$, i.e. $T_0=\lambda T_0$ and $-T_1=\lambda T_1$. We get $\lambda=1$, $T_1=0$ or $\lambda=-1$, $T_0=0$, i.e. $T\in \Cl^{\times (0)}\cup \Cl^{\times (1)}$.
\end{proof}

\begin{thm}  \label{th2} We have\footnote{According to the reviewer, this fact is well-known. We present this statement for the sake of completeness. The new results for the groups $\Gamma^k$ are presented in Section \ref{sectRelat}.}
$$\Gamma\subseteq \Gamma^k,\quad k=0, 1, \ldots, n.$$
As a consequence, we obtain
\begin{equation}
\Gamma=\bigcap_{k=0}^n \Gamma^k=\{T\in \Cl^\times:\quad T \Cl^{k}T^{-1}\subseteq \Cl^{k},\quad k=0, 1, \ldots, n\}.\nonumber
\end{equation}
\end{thm}

\begin{proof} Suppose $T=W v_1 \cdots v_m$, $W\in\Z^\times$, $v_j\in \Cl^{\times 1}$, $j=1, \ldots, m$. For $U_k\in \Cl^{k}$, we have
$$T U_k T^{-1}=W v_1 \ldots v_m U_k (W v_1 \ldots v_m)^{-1}=v_1 \ldots v_m U_k v_m^{-1} \ldots v_1^{-1}.$$
Since $(v_j)^2\in \Cl^{\times 0}$, we have $(v_j)^{-1}=\lambda_j v_j$ for some nonzero constant  $\lambda_j$, $j=1, \ldots, m$. Thus $\widetilde{(v_j)^{-1}}=v_j^{-1}$, $\widehat{(v_j)^{-1}}=-(v_j)^{-1}$, and we have
$$\widetilde{(v_j U_k v_j^{-1})}=v_j \widetilde{U_k} v_j^{-1},\qquad \widehat{(v_j U_k v_j^{-1})}=v_j \widehat{U_k} v_j^{-1}.$$
Using (\ref{quat}), we conclude that the elements $U_k$ and $v_j U_k v_j^{-1}$ are from the same subspace $\Cl^{\overline{k}}$. The grade of element $v_j U_k v_j^{-1}$ can take values between $k-2$ and $k+2$. Thus the grade equals $k$, and we get
$$v_j U_k v_j^{-1}\in \Cl^k,\qquad j=1, \ldots, m$$
and $T U_k T^{-1}\in \Cl^k$.
\end{proof}

We present new properties of the groups $\Gamma^k$, $k=1, \ldots, n$ in Section \ref{sectRelat} (see Lemmas \ref{lemmagammakq}, \ref{lemmagammaknk}, \ref{lemmagammakn}).

\section{The Lie groups $\P$}\label{sectP}

Let us use the following notation for the groups that preserve subspaces of fixed parity under similarity transformation
\begin{eqnarray}
\Gamma^{(k)}:&=&\{T\in \Cl^\times:\quad T \Cl^{(k)} T^{-1}\subseteq \Cl^{(k)}\},\quad k=0, 1.
\end{eqnarray}
Using Theorem \ref{th2}, we get
$$\Gamma\subseteq\Gamma^{(0)},\qquad \Gamma\subseteq\Gamma^{(1)}.$$
Let us consider the following group
\begin{eqnarray*}
\P:=\Z^\times(\Cl^{\times (0)}\cup \Cl^{\times (1)})=\left\lbrace
\begin{array}{ll}
\Cl^{\times (0)}\cup \Cl^{\times (1)}, & \mbox{if $n$ is even,}\\
\Cl^{\times 0n} \Cl^{\times (0)}, & \mbox{if $n$ is odd.}
\end{array}
\right.
\end{eqnarray*}
By Lemma \ref{lemmagammak}, we have $\Gamma^n=\P$ for even $n$.

We need the following lemma to prove Theorems \ref{th1} and \ref{th3}.

\begin{lem} \label{lemmacenter2} We have
\begin{eqnarray}
\{U\in \Cl:\quad [U, V]=0,\quad \forall V\in \Cl^{(0)}\}=\Cl^0\oplus \Cl^n.\label{center2}
\end{eqnarray}
\end{lem}
\begin{proof} Let us have $U e_{ab}=e_{ab} U$, $\forall a<b$. Taking $a=1$ and $b=2$, let us represent $U$ in the form $$U=A+e_1B+e_2 C+e_{12}D,$$ where the elements $A, B, C, D\in \Cl$ do not contain $e_1$ and $e_2$. We get $(A+e_1B+e_2 C+e_{12}D)e_{12}=e_{12}(A+e_1B+e_2 C+e_{12}D)$. Using
$A e_{12}=e_{12} A$, $D e_{12}=e_{12}D$, $e_1 B e_{12}=-e_{12}e_1B$, and $e_2C e_{12}=-e_{12}e_2C$,
we obtain $B=C=0$. Acting similarly for the other $a<b$, we obtain
$U=ae+b e_{1\ldots n}$ for some constants $a, b$.
\end{proof}

\begin{thm}\label{th1} The following three groups coincide\footnote{As one of the anonymous reviewers of this paper noted, this statement can be reformulated for the more general case of the graded central simple algebras (GCSAs).}
$$\P=\Gamma^{(0)}=\Gamma^{(1)}.$$
\end{thm}

\begin{proof}
Let us prove that $\P \subseteq \Gamma^{(0)}$ and $\P \subseteq \Gamma^{(1)}$. Suppose that $T\in \Z^\times(\Cl^{\times (0)}\cup \Cl^{\times (1)})$. Since $TT^{-1}=e$, we conclude that if $T=W T_0\in\Z^\times\Cl^{\times (0)}$, $W\in\Z^\times$, $T_0\in\Cl^{\times (0)}$, then $T^{-1}=W^{-1} T_0^{-1}\in\Z^\times\Cl^{\times (0)}$, and if $T=W T_1\in\Z^\times\Cl^{\times (1)}$, $W\in\Z^\times$, $T_1\in\Cl^{\times (1)}$, then $T^{-1}=W^{-1} T_1^{-1}\in\Z^\times\Cl^{\times (1)}$. Therefore $T \Cl^{(1)}T^{-1}\subseteq \Cl^{(1)}$ and $T \Cl^{(0)} T^{-1}\subseteq \Cl^{(0)}$.

Let us prove that $\Gamma^{(1)}\subseteq \P$. Suppose $T \Cl^{(1)} T^{-1}\subseteq \Cl^{(1)}$. Then we have
$$-T \Cl^{(1)} T^{-1}=\widehat{(T \Cl^{(1)} T^{-1})}=-\hat{T} \Cl^{(1)} \widehat{T^{-1}}.$$
Multiplying both sides on the right by $T$ and on the left by $\widehat{T^{-1}}$, we get
$$\widehat{T^{-1}} T \Cl^{(1)}=\Cl^{(1)} \widehat{T^{-1}}T,$$
thus $\widehat{T^{-1}}T\in\Z^\times$, and $\hat{T}=W T$ for some $W\in\Z^\times$. Taking grade involution, we obtain $T=\hat{W} \hat{T}$. Thus $T=\hat{W} W T$. If $n$ is even, then for $W=\lambda e$ with some constant $\lambda$, we get $\lambda=\pm 1$. We obtain $T\in \Cl^{\times (0)}\cup \Cl^{\times (1)}$. If $n$ is odd, then $W=ae+be_{1\ldots n}$ for some constants $a, b$. Suppose we have $T=T_0+T_1$, $T_0\in \Cl^{(0)}$, $T_1\in \Cl^{(1)}$. Then $T_0-T_1=(ae+be_{1\ldots n})(T_0+T_1)$, and $T_0=aT_0+be_{1\ldots n}T_1$. If $b=0$, then $a=1$, $W=e$, and $T\in \Cl^{(0)}\cup \Cl^{(1)}$. If $b\neq 0$, then $T_1=\mu e_{1 \ldots n} T_0$ for some constant $\mu$. Thus we have $T=T_0+T_1= (e+\mu e_{1\ldots n})T_0\in \Z^\times \Cl^{\times (0)}= \Z^\times (\Cl^{\times (0)}\cup \Cl^{\times (1)})$.

Let us prove that $\Gamma^{(0)}\subseteq\P$. Suppose $T \Cl^{(0)}T^{-1}\subseteq \Cl^{(0)}$. Then we have
$$T \Cl^{(0)}T^{-1}=\widehat{(T \Cl^{(0)}T^{-1})}=\hat{T} \Cl^{(0)} \widehat{T^{-1}}.$$
Multiplying both sides by $T$ on the right and by $\widehat{T^{-1}}$ on the left, we get
$$\widehat{T^{-1}} T \Cl^{(0)}=\Cl^{(0)} \widehat{T^{-1}}T.$$
Using Lemma \ref{lemmacenter2}, we obtain $\widehat{T^{-1}}T\in(\Cl^{0}\oplus \Cl^{n})^\times$, and $\hat{T}=W T$ for some $W\in \Cl^{\times 0 \, n}$. In the case of odd $n$, the proof is analogous to the proof of the previous case because $\Z=\Cl^{0\, n}$ in this case. Let us consider the case of even $n$. Suppose we have $T=T_0+T_1$, $T_0\in \Cl^{(0)}$, $T_1\in \Cl^{(1)}$. Then $T_0-T_1= W(T_0+T_1)$, i.e. $T_0=W T_0$ and $-T_1=WT_1$. We get $(W-e)T_0=0$ and $(W+e)T_1=0$. If at least one of two elements $W+e$ and $W-e$ is invertible, then we conclude that $T_0=0$ or $T_1=0$, and $T\in \Cl^{(0)}\cup \Cl^{(1)}$. Suppose that both elements $W+e$ and $W-e$ are non invertible. For $W=ae+be_{1\ldots n}$ with some constants $a, b$, we get $W\pm e=(a\pm 1)e+b e_{1\ldots n}$. We obtain
$$((a\pm 1)e+be_{1\ldots n})((a\pm 1)e-be_{1\ldots n})=(a\pm 1)^2 e-b^2 (e_{1\ldots n})^2 e.$$
We see that both elements can be non invertible only in the case $a=0$, $b=\pm 1$, $(e_{1\ldots n})^2=e$ or $a=0$, $b=\pm i$, $(e_{1\ldots n})^2=-e$ (in the case of the field $\C$). We get $W=\pm e_{1\ldots n}$ or $W=\pm i e_{1\ldots n}$. From $\hat{T}=WT$, we obtain
$$T=W\hat{T}=W\widehat{(T_0+T_1)}=W(T_0-T_1)=W T_0-W T_1=T_0 W+T_1 W=T W,$$
i.e. a contradiction, because $T$ is invertible and $W\neq e$.
\end{proof}

Note that we can prove $\Gamma^{(0)}=\Gamma^{(1)}$ in the case of odd $n$ in a simpler way. Multiplying both sides of $T \Cl^{(0)} T^{-1}\subseteq \Cl^{(0)}$ by $e_{1\ldots n}\in\Z$ and using $\Cl^{(0)}=e_{1\ldots n}\Cl^{(1)}$, we obtain $T \Cl^{(1)}T^{-1}\subseteq \Cl^{(1)}$. Similarly, in the opposite direction.

Note that we can also prove $\Gamma^{(1)}\subseteq \P$ in another way using the generalization of the Pauli's theorem (see \cite{DAN, Spin, genR}). Using $T \Cl^{(1)}T^{-1}\subseteq \Cl^{(1)}$, we get
\begin{eqnarray}
\beta_a:=Te_aT^{-1}\in \Cl^{(1)},\qquad a=1, \ldots, n.\label{beta}
\end{eqnarray}
The elements (\ref{beta}) satisfy $\beta_a \beta_b+\beta_b \beta_a=2\eta_{ab}e$. By the generalized Pauli's theorem for odd elements (see Theorems 5.1 and 5.2 in \cite{Spin}), we get $T\in \Z( \Cl^{\times (0)}\cup \Cl^{\times (1)})$ because the element $T$ equals
$$
\sum_{a_1\leq\cdots\leq a_k} \beta_{a_1\ldots a_k} F (e_{a_1\ldots a_k})^{-1}
$$
for some element $F\in \{e_{a_1\ldots a_k}\}$ up to an invertible element of the center $\Z^\times$.


Note that subspaces of fixed parity are direct sums of the subspaces (\ref{quat}):
$\Cl^{(0)}=\Cl^{\overline{02}}$ and $\Cl^{(1)}=\Cl^{\overline{13}}$. Below we study other subspaces, which are also direct sums of the subspaces (\ref{quat}): $\Cl^{\overline{01}}$, $\Cl^{\overline{23}}$, $\Cl^{\overline{03}}$, and $\Cl^{\overline{12}}$.

\section{The Lie groups $\A$}\label{sectA}

Let us use the following notation for the groups that preserve subspaces (\ref{quat}) or their direct sums under similarity transformation
\begin{eqnarray}
\Gamma^{\overline{k}}:&=&\{T\in \Cl^\times:\quad T \Cl^{\overline{k}} T^{-1}\subseteq \Cl^{\overline{k}}\},\quad k=0, 1, 2, 3,\\
\Gamma^{\overline{kl}}:&=&\{T\in \Cl^\times:\quad T \Cl^{\overline{kl}}T^{-1}\subseteq \Cl^{\overline{kl}}\},\quad k, l=0, 1, 2, 3.
\end{eqnarray}
Consider the following group\footnote{As one of the anonymous reviewers of this paper noted, the groups $\A, \B, \Q$ can be defined in the more general case of the graded central simple algebras with involution (GCSAsWI) and the corresponding statements can be reformulated for this more general case.}
\begin{eqnarray}
\A:=\{T\in \Cl^\times:\quad \tilde{T} T\in \Z^{\times}\}\label{3.3t}.
\end{eqnarray}
Note that the ``norm function'' $\psi(T):=\tilde{T}T$ is widely used in the theory of spin groups (see, for example \cite{BT, ABS, Lounesto, Bulg}). The spin groups
$$\Pin(p,q):=\{T\in\Gamma^\pm:\quad \psi(T)=\pm e\},\qquad \Spin(p,q):=\Pin(p,q)\cap \Cl^{(0)}$$
are defined as normalized subgroups of the Lipschitz group
\begin{eqnarray}
\Gamma^\pm:=\{T\in \Cl^{\times (0)}\cup \Cl^{\times (1)}:\quad T \Cl^{1}T^{-1}\subseteq \Cl^{1}\}\subseteq \Gamma.
\end{eqnarray}
It is well-known that $\psi: \Gamma^\pm\to \Cl^{\times 0}$. The group $\A$ contains the groups $\Gamma$ and $\Gamma^\pm$ as subgroups.

\begin{lem} The norm function $\tilde{T}T$ takes values in $\Cl^{\overline{01}}$, i.e.
$$\tilde{T}T\in \Cl^{\overline{01}}\qquad \forall T\in \Cl.$$
As a consequence, we get
\begin{eqnarray}
\A=\left\lbrace
\begin{array}{ll}
\{T\in \Cl^\times:\quad \tilde{T} T\in \Cl^{\times 0}\}, & \mbox{if $n=0, 2, 3 \mod 4$,}\\
\{T\in \Cl^\times:\quad \tilde{T} T\in (\Cl^{0}\oplus \Cl^{n})^\times\}, & \mbox{if $n=1\mod 4$.}
\end{array}
\right.\label{tt4}
\end{eqnarray}
\end{lem}
\begin{proof}
We have $\widetilde{\tilde{T}T}=\tilde{T} \tilde{\tilde{T}}=\tilde{T}T$ for any $T\in \Cl$. We conclude that the reversion does not change the expression $\tilde{T}T$. Using (\ref{quat}), we obtain $\tilde{T}T\in \Cl^{\overline{01}}$. Using (\ref{3.3t}), we get (\ref{tt4}).
\end{proof}
\begin{thm}\label{th3t} The following three groups coincide
$$\A=\Gamma^{\overline{01}}=\Gamma^{\overline{23}}.$$
\end{thm}

\begin{proof} Let us prove that $\A\subseteq\Gamma^{\overline{01}}$ and $\A\subseteq\Gamma^{\overline{23}}$. Let us have $\tilde{T} T=W\in\Z^\times$. For $U\in \Cl^{\overline{01}}$, we have
$$\widetilde{TUT^{-1}}= \widetilde{T^{-1}} U \tilde{T}=TW^{-1}U W T^{-1}= TUT^{-1}$$
i.e. $TUT^{-1}\in \Cl^{\overline{01}}$.
For $U\in \Cl^{\overline{23}}$, we have
$$\widetilde{TUT^{-1}}= -\widetilde{T^{-1}} U \tilde{T}=-TW^{-1}U WT^{-1}= -TUT^{-1}$$
i.e. $TUT^{-1}\in \Cl^{\overline{23}}$.

Let us prove that $\Gamma^{\overline{01}}\subseteq \A$. For $U\in \Cl^{\overline{01}}$, we have
$$TUT^{-1}= \widetilde{TUT^{-1}}= \widetilde{T^{-1}} U \tilde{T}.$$
Multiplying both sides on the right by $T$ and on the left by $\tilde{T}$, we get
$$(\tilde{T}T)U=U (\tilde{T}T)\qquad \forall U\in \Cl^{\overline{01}}.$$
In particular, we conclude that $\tilde{T}T$ commutes with all generators $e_a\in \Cl^{\overline{01}}$, $a=1, \ldots, n $. Thus $\tilde{T}T\in\Z^\times$.

Let us prove that $\Gamma^{\overline{23}}\subseteq \A$. For $U\in \Cl^{\overline{23}}$, we have
$$-TUT^{-1}= \widetilde{TUT^{-1}}= -\widetilde{T^{-1}} U \tilde{T}.$$
We get
$$(\tilde{T}T)U=U (\tilde{T}T)\qquad \forall U\in \Cl^{\overline{23}}.$$
If $n\geq 3$, then we can always represent each generator $e_a$, $a=1,\ldots, n$ as the product of elements of grade 2 and 3. This implies that $\tilde{T}T$ commutes with all generators $e_a$, $a=1, \ldots, n $. For example, if $\tilde{T}T$ commutes with $e_{123}$ and $e_{23}$, then it commutes with $e_1$. We obtain $\tilde{T}T\in\Z^\times$. If $n=2$, then from $[\tilde{T}T, e_{12}]=0$, we obtain $\tilde{T}T\in \Cl^0\oplus \Cl^2$. Using $\tilde{T}T\in \Cl^{\overline{01}}$, we get $\tilde{T}T\in \Cl^{\times 0}=\Z^\times$.
\end{proof}

\section{The Lie groups $\B$}\label{sectB}

Let us consider the following group
\begin{eqnarray}
\B:=\{T\in \Cl^\times:\quad \hat{\tilde{T}} T\in \Z^\times\}\label{4.3t}.
\end{eqnarray}
The norm function $\chi(T):=\hat{\tilde{T}}T$ as well as the function $\psi(T)$ (see above) is widely used in the theory of spin groups (see \cite{BT, ABS, Lounesto, Bulg}).

\begin{lem} The norm function $\hat{\tilde{T}}T$ takes values in $\Cl^{\overline{03}}$, i.e.
$$\hat{\tilde{T}}T\in \Cl^{\overline{03}}\qquad \forall T\in \Cl.$$
As a consequence, we get
\begin{eqnarray}
\B=\left\lbrace
\begin{array}{ll}
\{T\in \Cl^\times:\quad \hat{\tilde{T}} T\in \Cl^{\times 0}\}, & \mbox{if $n=0, 1, 2 \mod 4$,}\\
\{T\in \Cl^\times:\quad \hat{\tilde{T}} T\in (\Cl^{0}\oplus \Cl^{n})^\times\}, & \mbox{if $n=3\mod 4$.}
\end{array}
\right.\label{tt5}
\end{eqnarray}
In the particular cases, we have
$$\B= \Cl^\times,\qquad n\leq 3.$$
\end{lem}
\begin{proof}
We have
$\widehat{\widetilde{\hat{\tilde{T}}T}}=\hat{\tilde{T}} \hat{\tilde{\hat{\tilde{T}}}}=\hat{\tilde{T}}T$. We conclude that the operation $\hat{\tilde{\quad}}$ does not change the expression $\hat{\tilde{T}}T$.
Using (\ref{quat}), we obtain $\hat{\tilde{T}}T\in \Cl^{\overline{03}}$. Using (\ref{4.3t}), we get (\ref{tt5}). In the cases $n\leq 3$, the condition $\hat{\tilde{T}} T\in \Z^\times=\Cl^{\overline{03}}$ holds automatically.
\end{proof}

\begin{thm}\label{th4t} The following three groups coincide
$$\B=\Gamma^{\overline{03}}=\Gamma^{\overline{12}}.$$
\end{thm}

\begin{proof} Let us prove that $\B\subseteq\Gamma^{\overline{03}}$ and $\B\subseteq\Gamma^{\overline{12}}$. Let us have $\hat{\tilde{T}} T=W\in\Z^\times$. For $U\in \Cl^{\overline{03}}$, we have
$$\widehat{\widetilde{TUT^{-1}}}= \widehat{\widetilde{T^{-1}}} U \tilde{\hat{T}}= T W^{-1}U WT^{-1}=TUT^{-1},$$
i.e. $TUT^{-1}\in \Cl^{\overline{03}}$. For $U\in \Cl^{\overline{12}}$, we have
$$\widehat{\widetilde{TUT^{-1}}}= -\widehat{\widetilde{T^{-1}}} U \tilde{\hat{T}}=- T W^{-1}U W T^{-1} =-TUT^{-1},$$
i.e. $TUT^{-1}\in \Cl^{\overline{12}}$.

Let us prove that $\Gamma^{\overline{12}}\subseteq\B$. For $U\in \Cl^{\overline{12}}$, we have
$$-TUT^{-1}= \widehat{\widetilde{TUT^{-1}}}=-\widehat{\widetilde{T^{-1}}} U \tilde{\hat{T}}.$$
Multiplying both sides on the right by $T$ and on the left by $\tilde{\hat{T}}$, we get
$$(\tilde{\hat{T}}T)U=U (\tilde{\hat{T}}T)\qquad \forall U\in \Cl^{\overline{12}}.$$
In particular, we conclude that $\tilde{\hat{T}}T$ commutes with all generators $e_a\in \Cl^{\overline{12}}$, $a=1, \ldots, n$. Thus $\tilde{\hat{T}}T\in \Z^\times$.

Let us prove that $\Gamma^{\overline{03}}\subseteq\B$. For $U\in \Cl^{\overline{03}}$, we have
$$TUT^{-1}= \widehat{\widetilde{TUT^{-1}}}= \widehat{\widetilde{T^{-1}}} U \tilde{\hat{T}}.$$
Multiplying both sides on the right by $T$ and on the left by $\tilde{\hat{T}}$, we get
$$(\tilde{\hat{T}}T)U=U (\tilde{\hat{T}}T)\qquad \forall U\in \Cl^{\overline{03}}.$$
For $n\leq 3$, we get $\Gamma^{\overline{03}}=\B= \Cl^\times$. If $n\geq 4$, then we can always represent each generator $e_a$, $a=1, \ldots, n$, as the product of two elements of grades 3 and 4. This implies that $\tilde{\hat{T}}T$ commutes with all generators $e_a$, $a=1,\ldots, n$. For example, if $\tilde{\hat{T}}$ commutes with $e_{234}$ and $e_{1234}$, then it commutes with $e_1$. We get $\tilde{\hat{T}}T\in \Z^\times$.
\end{proof}

\section{The Lie groups $\Q$ and $\Q^\prime$}\label{sectQ}

Let us consider the following group
\begin{eqnarray}
\Q&:=&\{T\in \Z^\times(\Cl^{\times (0)}\cup \Cl^{\times (1)}):\quad \tilde{T} T\in \Z^\times\}.\label{Qpq}
\end{eqnarray}

\begin{lem}\label{lemQ} We have
\begin{eqnarray}
\Q=\A\cap\P=\B\cap\P=\A\cap \B,\qquad \Q\subseteq \P,\qquad \Q\subseteq \A,\qquad \Q\subseteq \B.\label{yyy4}
\end{eqnarray}
In the particular cases, we have
\begin{eqnarray}
&&\Q=\P=\Z^\times(\Cl^{\times (0)}\cup \Cl^{\times (1)}),\qquad n\leq 3;\label{rty}\\
&&\Q\neq \P,\qquad \A\neq \P,\qquad n=4.\label{rty2}
\end{eqnarray}
\end{lem}
Note that below (see Lemma \ref{lem7}), using auxiliary statements, we also prove that $\Q=\P=\A$ in the cases $n\leq 3$, and $\Q\neq\A$ in the case $n=4$.
\begin{proof} The first two statements (\ref{yyy4}) are trivial because of the definition (\ref{Qpq}). Let us prove the nontrivial statement $\A\cap\B=\Q$. Suppose we have element $T\in \Cl^\times$ such that $\tilde{T}T=W_1\in\Z^\times$ and $\hat{\tilde{T}} T=W_2\in\Z^\times$. We get $\hat{T}=W T$ for some $W\in\Z^\times$. In the proof of Theorem \ref{th1}, we have already shown that this implies $T\in\Z^\times(\Cl^{\times (0)}\cup \Cl^{\times (1)})$. Thus $\A\cap\B=\Q$.

For $n\leq 3$, we have $\Cl^{\overline{0}}=\Cl^0$. Using $\tilde{T}T\in \Cl^{\overline{01}}$ and $\hat{\tilde{T}}T\in \Cl^{\overline{03}}$, we get $\Q=\P$ for $n\leq 3$.

In the case $n=4$, the element
$$T=e_{12}+2e_{34}\in \Cl^{(0)}$$
is invertible because
$$(e_{12}+2e_{34})(e_{12}-2e_{34})=(e_{12})^2-4(e_{34})^2$$
is a nonzero scalar. Also we have
$$\tilde{T} T=-(e_{12}+2e_{34})(e_{12}+2e_{34})=-(e_{12})^2-4(e_{34})^2-4e_{1234}\notin \Z^\times,$$
i.e. $T\in\P$, $T\notin\A$, and $T\notin\Q$. Thus $\P\neq \A$ and $\P\neq \Q$ in the case $n=4$.
\end{proof}

Let us consider the group
\begin{eqnarray}
\Q^\prime:=\{T\in \Z^\times(\Cl^{\times (0)}\cup \Cl^{\times (1)}):\quad \tilde{T} T\in (\Cl^{0}\oplus \Cl^n)^\times\}.
\end{eqnarray}
This group coincides with the group $\Q$ in the cases $n=1, 3\mod 4$ because $\Z=\Cl^0\oplus \Cl^n$ for odd $n$, and in the case $n=2\mod 4$ because $\tilde{T}T\in \Cl^{\overline{01}}$, $\hat{\tilde{T}}T\in \Cl^{\overline{03}}$:
$$\Q^\prime=\Q,\qquad n=1, 2, 3\mod 4.$$
Let us consider the group $\Q^\prime$ in the case $n=0\mod 4$.

\begin{lem}\label{lemQQ} We have
\begin{equation}
\Q\subseteq \Q^\prime \subseteq \P,\qquad n=0\mod 4,\label{QQprime}
\end{equation}
where
\begin{eqnarray}
&&\Q\neq \Q^\prime,\qquad n=4, 8, 12, \ldots; \label{p1}\\
&&\Q^\prime=\P,\qquad n=4;\label{p2}\\
&&\Q^\prime\neq \P,\qquad n=8, 12, 16, \ldots\label{p3}
\end{eqnarray}
For the groups
\begin{eqnarray}
\A^\prime&:=&\{T\in \Cl^\times:\quad \tilde{T}T\in(\Cl^{0}\oplus \Cl^n)^\times\},\\
\B^\prime&:=&\{T\in \Cl^\times:\quad \hat{\tilde{T}}T\in(\Cl^{0}\oplus \Cl^n)^\times\},
\end{eqnarray}
we have
\begin{eqnarray}
\A^\prime\cap\B^\prime=\Q^\prime.\label{ABQprime}
\end{eqnarray}
\end{lem}
\begin{proof} We obtain (\ref{QQprime}) using the definitions of the corresponding groups.

Let us prove (\ref{p1}). For the element $T=e+2e_{1\ldots n}$, we have
$$\tilde{T}T=(e+2e_{1\ldots n})(e+2e_{1\ldots n})=e+4(e_{1\ldots n})^2+4e_{1\ldots n}\in \Cl^{\times 0 n},$$
i.e. $T\in\Q^\prime$ and $T\notin\Q$. The element $T$ is invertible because
$$
(e+2e_{1\ldots n})(e-2e_{1\ldots n})=e-4(e_{1\ldots n})^2
$$
is a nonzero scalar.

We have (\ref{p2}) because the condition $\tilde{T}T\in \Cl^{\overline{0}}=\Cl^0\oplus \Cl^4$ holds automatically in the case $n=4$.

Let us prove (\ref{p3}). The element $T=e_{12}+2e_{34}\in\Cl^{(0)}$ is invertible
because
$$(e_{12}+2e_{34})(e_{12}-2e_{34})=(e_{12})^2-4(e_{34})^2$$
is a nonzero scalar. Also we have
$$\tilde{T} T=-(e_{12}+2e_{34})(e_{12}+2e_{34})=-(e_{12})^2-4(e_{34})^2-4e_{1234}\notin \Cl^{\times 0 n},$$
i.e. $T\in\P$ and $T\notin\Q^\prime$.

Let us prove (\ref{ABQprime}). If $n=1, 2, 3\mod 4$, then $\A^\prime=\A$, $\B^\prime=\B$, and $\Q^\prime=\Q$, thus $\Q^\prime=\Q=\A\cap\B=\A^\prime\cap\B^\prime$. Let us consider the case $n=0\mod 4$. Suppose we have an element $T\in \Cl^\times$ such that
$$
\tilde{T}T=W_1\in(\Cl^{0}\oplus \Cl^n)^\times,\qquad \hat{\tilde{T}} T=W_2\in(\Cl^{0}\oplus \Cl^n)^\times.
$$
We get $\hat{T}=W T$ for some $W\in(\Cl^{0}\oplus \Cl^n)^\times$. In the proof of Theorem \ref{th1}, we have already shown that this implies $T\in \Cl^{\times (0)}\cup \Cl^{\times (1)}$.
\end{proof}

\begin{thm}\label{th3} In the cases $n\geq 4$, we have
\begin{eqnarray}
&&\Q=\Gamma^{\overline{1}}=\Gamma^{\overline{3}}\neq
\Q^\prime=\Gamma^{\overline{0}}=\Gamma^{\overline{2}},\qquad n=0\mod 4,\\
&&\Q=\Gamma^{\overline{0}}=\Gamma^{\overline{1}}=
\Gamma^{\overline{2}}=\Gamma^{\overline{3}},\qquad n=1, 2, 3\mod 4.
\end{eqnarray}
In the exceptional cases, we have
\begin{eqnarray}
&&\Gamma^{\overline{0}}=\Cl^\times\neq \Gamma^{\overline{1}}=\Gamma^{\overline{2}}=\Q=\P= \Cl^{\times (0)}\cup \Cl^{\times (1)},\quad n=2,\nonumber\\
&&\Gamma^{\overline{0}}=\Gamma^{\overline{3}}=\Cl^\times\neq \Gamma^{\overline{1}}=\Gamma^{\overline{2}}=\Q=\P=\Z^\times \Cl^{\times (0)},\quad n=3.\nonumber
\end{eqnarray}
As a consequence, we obtain (the analogue of Theorem \ref{th2} for grades)
\begin{eqnarray}
&&\Gamma^{\overline{1}}\subseteq \Gamma^{\overline{k}},\quad k=0, 1, 2, 3,\qquad \bigcap_{k=0}^3\Gamma^{\overline{k}}=\Gamma^{\overline{1}},\quad \mbox{i.e.}\\
&&\Gamma^{\overline{1}}=\Q=\{T\in \Cl^\times:\quad T \Cl^{\overline{k}}T^{-1}\subseteq \Cl^{\overline{k}},\quad k=0, 1, 2, 3\}.\nonumber
\end{eqnarray}
\end{thm}

\begin{proof} Let us prove that $\Q\subseteq\Gamma^{\overline{k}}$, $k=0, 1, 2, 3$. Suppose $T\in\Z^\times(\Cl^{\times (0)}\cup \Cl^{\times (1)})$ and $\tilde{T} T=W\in \Z^\times$. For an arbitrary $U_{\overline{k}}\in \Cl^{\overline k}$, $\overline{k}=\overline{0}, \overline{1}, \overline{2}, \overline{3}$, we have
\begin{eqnarray}
&&\widetilde{T U_{\overline{k}} T^{-1}}=\widetilde{T^{-1}} \widetilde{U_{\overline{k}}} \tilde{T} =W^{-1} T \widetilde{U_{\overline{k}}} T^{-1}W=T \widetilde{U_{\overline{k}}} T^{-1},\nonumber\\
&&\widehat{TU_{\overline{k}} T^{-1}}=\hat{T} \widehat{U_{\overline{k}}} \widehat{T^{-1}} =T \widehat{U_{\overline{k}}} T^{-1}.\nonumber
\end{eqnarray}
Using (\ref{quat}), we conclude that the elements $T U_{\overline{k}} T^{-1}\in\Cl^{\overline k}$. Thus $\Q\subseteq\Gamma^{\overline{k}}$, $k=0, 1, 2, 3$.

Let us prove that in the cases $n=0\mod 4$, we have $\Q^\prime\subseteq \Gamma^{\overline{0}}$, $\Q^\prime\subseteq \Gamma^{\overline{2}}$. Suppose $T\in \Cl^{\times (0)}\cup \Cl^{\times (1)}$ and $\tilde{T} T=W\in (\Cl^0\oplus \Cl^n)^\times$. For an arbitrary $U_{\overline{k}}\in \Cl^{\overline k}$, we have
\begin{eqnarray}
&&\widetilde{T U_{\overline{k}} T^{-1}}=\widetilde{T^{-1}} \widetilde{U_{\overline{k}}} \tilde{T} = T W^{-1} \widetilde{U_{\overline{k}}}  W T^{-1}=T \widetilde{U_{\overline{k}}} T^{-1},\qquad k=0, 2,\nonumber\\
&&\widehat{TU_{\overline{k}} T^{-1}}=\hat{T} \widehat{U_{\overline{k}}} \widehat{T^{-1}} =T \widehat{U_{\overline{k}}} T^{-1},\qquad k=0, 1, 2, 3,\nonumber
\end{eqnarray}
where we use that $W$ commutes with even elements. Using (\ref{quat}), we conclude that $T U_{\overline{k}} T^{-1}\in\Cl^{\overline k}$ for $k=0$ and $2$. Thus $\Q^\prime\subseteq\Gamma^{\overline{k}}$, $k=0, 2$.


Let us prove that $\Gamma^{\overline{1}}\subseteq \Q$. Suppose $T \Cl^{\overline 1}T^{-1}\subseteq \Cl^{\overline 1}$. From $T e_a T^{-1}\in \Cl^{\overline 1}$, $a=1, \ldots, n$, we get
$$(\tilde{T})^{-1}e_a \tilde{T}=\widetilde{(T e_a T^{-1})}= Te_a T^{-1}.$$
Multiplying both sides on the right by $T$ and on the left by $\tilde{T}$, we get
$$e_a (\tilde{T} T)=(\tilde{T} T) e_a,\qquad a=1, \ldots, n,$$
i.e. $\tilde{T} T\in \Z^\times$, and $\Gamma^{\overline{1}}\subseteq \A$.
Also we have
$$-(\hat{\tilde{T}})^{-1}e_a \hat{\tilde{T}}=\widehat{\widetilde{(T e_a T^{-1})}}= -Te_a T^{-1}.$$
Multiplying both sides on the right by $T$ and on the left by $\hat{\tilde{T}}$, we get
$$e_a (\hat{\tilde{T}} T)=(\hat{\tilde{T}} T) e_a,\qquad a=1, \ldots, n,$$
i.e. $\hat{\tilde{T}} T\in \Z^\times$, and $\Gamma^{\overline{1}}\subseteq \B$.
Using Lemma \ref{lemQ}, we obtain $\Gamma^{\overline{1}}\subseteq \A\cap\B=\Q$.

In the exceptional case $n=3$, we have $\Gamma^{\overline{3}}=\Cl^\times$ because $\Cl^{\overline{3}}=\Cl^3\in\Z$ in this case.

Let us prove that $\Gamma^{\overline{3}}\subseteq \Q$, $n\geq 4$. Suppose $T \Cl^{\overline 3}T^{-1}\subseteq \Cl^{\overline 3}$. We obtain
$$-(\tilde{T})^{-1} U\tilde{T}=\widetilde{(T U T^{-1})}= -T U T^{-1},\qquad \forall U\in \Cl^{\overline 3},$$
and
$$U (\tilde{T} T)=(\tilde{T} T) U,\qquad \forall U\in \Cl^{\overline 3}.$$
If $n\geq 4$, then we can always represent each generator $e_a$, $a=1,\ldots, n$, as the product of 3 elements from the subspace $\Cl^{\overline 3}$. This implies that $\tilde{T}T$ commutes with all generators $e_a$, $a=1, \ldots, n $. For example, if $\tilde{T}T$ commutes with $e_{123}$, $e_{124}$, $e_{134}$, then it is trivial that it commutes with $e_1$. We obtain $\tilde{T} T\in \Z^\times$, i.e. $\Gamma^{\overline{3}}\subseteq \A$. Analogously, we get
$$(\hat{\tilde{T}})^{-1} U\hat{\tilde{T}}=\widehat{\widetilde{(T U T^{-1})}}= T U T^{-1},\qquad \forall U\in \Cl^{\overline 3},$$
and
$$U (\hat{\tilde{T}} T)=(\hat{\tilde{T}} T) U,\qquad \forall U\in \Cl^{\overline 3}.$$
Finally, for $n\geq 4$, we obtain $\hat{\tilde{T}}T\in\Z^\times$, and $\Gamma^{\overline{3}}\subseteq \B$. Thus $\Gamma^{\overline{3}}\subseteq \A\cap\B=\Q$ in the cases $n\geq 4$.

Let us prove that $\Gamma^{\overline{2}}\subseteq \Q$ in the cases $n=1, 2, 3\mod 4$ and $\Gamma^{\overline{2}}\subseteq \Q^\prime$ in the case $n=0\mod 4$. Suppose $T \Cl^{\overline 2}T^{-1}\subseteq \Cl^{\overline 2}$. We have
$$-(\tilde{T})^{-1} U\tilde{T}=\widetilde{(T U T^{-1})}= -T U T^{-1},\qquad \forall U\in \Cl^{\overline 2}.$$
Therefore
$$U (\tilde{T} T)=(\tilde{T} T) U,\qquad \forall U\in \Cl^{\overline 2}.$$
This implies
$$U (\tilde{T} T)=(\tilde{T} T) U,\qquad \forall U\in \Cl^{(0)}$$
because we can always ($n\geq 2$) represent all even basis elements of Clifford algebra as the products of elements from the subspace $\Cl^{\overline 2}$. Using Lemma \ref{lemmacenter2}, we get $\tilde{T} T\in \Cl^0\oplus \Cl^n$.  Analogously,
$$-(\hat{\tilde{T}})^{-1} U\hat{\tilde{T}}=\widehat{\widetilde{(T U T^{-1})}}= -T U T^{-1},\qquad \forall U\in \Cl^{\overline 2}.$$
Therefore
$$U (\hat{\tilde{T}} T)=(\hat{\tilde{T}} T) U,\qquad \forall U\in \Cl^{\overline 2}.$$
This implies
$$U (\hat{\tilde{T}} T)=(\hat{\tilde{T}} T) U,\qquad \forall U\in \Cl^{(0)}.$$
Using Lemma \ref{lemmacenter2}, we obtain $\hat{\tilde{T}} T\in \Cl^0\oplus \Cl^n$. We have $\tilde{T}T\in \Cl^{\overline{01}}$, $\hat{\tilde{T}}T\in \Cl^{\overline{03}}$. Thus in the cases $n=1, 2, 3\mod 4$, we get $\tilde{T} T\in \Cl^0$ and $\hat{\tilde{T}} T\in \Cl^0$. We obtain the group $\A\cap\B=\Q$ in this case. In the case $n=0\mod 4$, we obtain the group $\A^\prime\cap\B^\prime=\Q^\prime$.

In the exceptional cases $n=2, 3$, we have $\Gamma^{\overline{0}}=\Gamma^{0}=\Cl^\times$ because $\Cl^{\overline{0}}=\Cl^{0}\in\Z$ in this case.

Let us prove that $\Gamma^{\overline{0}}\subseteq \Q$ in the cases $n=1, 2, 3\mod 4$, $n\geq 5$, and $\Gamma^{\overline{0}}\subseteq \Q^\prime$ in the cases $n=0\mod 4$, $n\geq 4$. Suppose $T \Cl^{\overline 0}T^{-1}\subseteq \Cl^{\overline 0}$. We have
$$(\tilde{T})^{-1} U\tilde{T}=\widetilde{(T U T^{-1})}= T U T^{-1},\qquad \forall U\in \Cl^{\overline 0}.$$
Therefore
\begin{equation}
U (\tilde{T} T)=(\tilde{T} T) U,\qquad \forall U\in \Cl^{\overline 0}.\label{54}
\end{equation}
In the cases $n\geq 5$, this implies
$$U (\tilde{T} T)=(\tilde{T} T) U,\qquad \forall U\in \Cl^{(0)}$$
because we can always represent each basis element of grade 2 as the product of elements from the subspace $\Cl^{\overline 0}$. For example, if $\tilde{T}T$ commutes with $e_{1345}$ and $e_{2345}$, then it commutes with $e_{12}$. Using Lemma \ref{lemmacenter2}, we get
\begin{equation}
\tilde{T} T\in \Cl^0\oplus \Cl^n.\label{55}
\end{equation}
If $n=4$, then $\Cl^{\overline 0}=\Cl^0\oplus \Cl^4$ and, from (\ref{54}), we get $\tilde{T}T\in \Cl^{(0)}=\Cl^{024}$. Using $\tilde{T}T\in \Cl^{\overline{0}}$, we get $\tilde{T}T\in \Cl^{04}$, i.e. (\ref{55}) again. Analogously, we obtain
$$(\hat{\tilde{T}})^{-1} U\hat{\tilde{T}}=\widehat{\widetilde{(T U T^{-1})}}= T U T^{-1},\qquad \forall U\in \Cl^{\overline 0}.$$
Therefore
\begin{equation}
U (\hat{\tilde{T}} T)=(\hat{\tilde{T}} T) U,\qquad \forall U\in \Cl^{\overline 0}.\label{57}
\end{equation}
In the cases $n\geq 5$, this implies
$$U (\hat{\tilde{T}} T)=(\hat{\tilde{T}} T) U,\qquad \forall U\in \Cl^{(0)}.$$
Using Lemma \ref{lemmacenter2}, we get
\begin{equation}
\hat{\tilde{T}} T\in \Cl^0\oplus \Cl^n.\label{56}
\end{equation}
In the case $n=4$, from (\ref{57}), we get again (\ref{56}). Finally, in the cases $n=1, 2, 3\mod 4$, for $n\geq 4$, using $\tilde{T}T\in \Cl^{\overline{01}}$, $\hat{\tilde{T}}T\in \Cl^{\overline{03}}$, we get from (\ref{55}) and (\ref{56}) the conditions $\tilde{T}T\in \Cl^0$ and $\hat{\tilde{T}} T\in \Cl^0$. We obtain the group $\A\cap\B=\Q$ in this case. In the cases $n=0\mod 4$, $n\geq 4$, from the conditions (\ref{55}) and (\ref{56}) we obtain the group $\A^\prime\cap\B^\prime=\Q^\prime$.
\end{proof}

Let us give one example for the case $n=4$, $p=1$, $q=3$. The element $T=e+e_{1234}\in \Gamma^{\overline{2}}=\Q^\prime$ and $T\notin \Gamma^{\overline{1}}=\Q$. We have $(e_{1234})^2=-e$ and $T^{-1}=\frac{1}{2}(e-e_{1234})$. Thus
$$T e_1 T^{-1}=\frac{1}{2}(e+e_{1234})e_1 (e-e_{1234})=-e_{234}\notin \Cl^{\overline{1}},$$
and for an arbitrary $U\in \Cl^{\overline{2}}$
\begin{eqnarray}
T U T^{-1}&=&\frac{1}{2}(e+e_{1234})U (e-e_{1234})\nonumber\\
&=&\frac{1}{2}(U -U e_{1234}+e_{1234}U - e_{1234}U e_{1234})=U \in \Cl^{\overline{2}}.\nonumber
\end{eqnarray}
We get $\Q=\Gamma^{\overline{1}}\neq\Gamma^{\overline{2}}=\Q^\prime$ in this case.

\section{The relations between groups}\label{sectRelat}

In this section, we prove some new properties of the groups $\Gamma^k$, $\Q$, $\Q^\prime$ and $\A$. In particular, we show that these groups are closely related in the cases of small dimensions.

\begin{lem}\label{lemmagammakq} We have
\begin{eqnarray}
\Gamma^k&\subseteq& \Q,\qquad k=1, 2, 3, \ldots, n-1,\qquad n=1, 2, 3\mod 4,\label{yyy7}\\
\Gamma^k&\subseteq& \Q,\qquad k=1, 3, 5, \ldots, n-1,\qquad n=0\mod 4,\label{yyy8}\\
\Gamma^k&\subseteq& \Q^\prime,\qquad k=2, 4, 6, \ldots, n-2,\qquad n=0\mod 4.\label{yyy9}
\end{eqnarray}
\end{lem}
Below (see Theorem \ref{thCap} and Lemma \ref{lemmagammakn}) we prove that we have equality in (\ref{yyy7}), (\ref{yyy8}), and (\ref{yyy9}) in the cases $n\leq 5$.
\begin{proof}
Let us have $TU_kT^{-1}\subseteq \Cl^k$ for any $U_k\in\Cl^k$ with some fixed $k$. We get
$$\pm TU_kT^{-1}=\widetilde{TU_k T^{-1}}=\pm \widetilde{T^{-1}} U_k \widetilde{T},\qquad \forall U_k\in\Cl^k.$$
Multiplying both sides by $T$ on the right and by $\tilde{T}$ on the left, we get
$$ U_k(\tilde{T}T)=(\tilde{T}T)U_k,\qquad \forall U_k\in\Cl^k.$$
If $k$ is odd and $k\neq n$, then we can always represent each generator $e_a$, $a=1, \ldots, n$ as the product of elements of grade $k$. We obtain
$$ e_a(\tilde{T}T)=(\tilde{T}T)e_a,\qquad a=1, \ldots, n,$$
i.e. $\tilde{T}T\in\Z^\times$.

If $k$ is even and $k\neq n$, $k\neq 0$, then we can always represent each basis element of grade 2 as the product of elements of grade $k$. We obtain
$$ e_{ab}(\tilde{T}T)=(\tilde{T}T)e_{ab},\qquad \forall a<b.$$
Using Lemma \ref{lemmacenter2}, we get $\tilde{T}T\in\Cl^0\oplus\Cl^n$, which is equivalent to $\tilde{T}T\in\Cl^0$ in the cases $n=1, 2, 3\mod 4$.

Also we have
$$\pm TU_k T^{-1}=\widehat{\widetilde{TU_k T^{-1}}}=\pm \widehat{\widetilde{T^{-1}}} U_k \hat{\tilde{T}},\qquad \forall U_k\in\Cl^k.$$
Multiplying both sides by $T$ on the right and by $\hat{\tilde{T}}$ on the left, we get
$$ U_k(\hat{\tilde{T}}T)=(\hat{\tilde{T}}T)U_k,\qquad \forall U_k\in\Cl^k.$$
If $k$ is odd and $k\neq n$, then we obtain $\hat{\tilde{T}}T\in\Z^\times$.  If $k$ is even and $k\neq n$, $k\neq 0$, then we obtain $\hat{\tilde{T}}T\in\Cl^0\oplus\Cl^n$, which is equivalent to $\hat{\tilde{T}}T\in\Cl^0$ in the cases $n=1, 2, 3\mod 4$.

Finally, in the cases $n=1, 2, 3\mod 4$, we obtain the group $\Q=\A\cap\B$. In the case $n=0\mod 4$, we obtain the group $\Q=\A\cap\B$ for odd $k$ and the group $\Q^\prime=\A^\prime\cap \B^\prime$ for even $k$.
\end{proof}

Using previous Lemma, we conclude that $\Gamma^k\subseteq \P$ because $\Q\subseteq\P$, $\Q^\prime\subseteq\P$ (see Lemmas \ref{lemQ} and \ref{lemQQ}). Thus we can rewrite the definitions of the groups $\Gamma^k$ (\ref{gammak}) in the form
$$
\Gamma^k:=\{T\in\Z^\times(\Cl^{\times (0)}\cup\Cl^{\times (1)}):\, T \Cl^{k}T^{-1}\subseteq\Cl^{k}\}
,\quad k=1, \ldots, n-1.
$$

\begin{lem}  \label{lemmagammaknk} We have
$$
\Gamma^k=\Gamma^{n-k},\quad k=1, \ldots, n-1.
$$
\end{lem}

\begin{proof} If $n$ is odd, then $T \Cl^k T^{-1}\subseteq \Cl^k$ is equivalent to $T \Cl^{n-k} T^{-1}\subseteq \Cl^{n-k}$, $k=1, \ldots, n-1$, because we can multiply both sides of this conditions by $e_{1\ldots n}\in\Z$ and use $\Cl^k e_{1\ldots n}=\Cl^{n-k}$.

If $n$ is even, then $T \Cl^k T^{-1}\subseteq \Cl^k$ is also equivalent to $T \Cl^{n-k}T^{-1}\subseteq \Cl^{n-k}$, $k=1, \ldots, n-1$. We can multiply both sides of this conditions by $e_{1\ldots n}$ and use that $T\in\Cl^{\times (0)}\cup\Cl^{\times (1)}$ (see the note before this Lemma), use that $e_{1\ldots n}$ commutes with all even elements and anticommutes with all odd elements, and $\Cl^k e_{1\ldots n}=\Cl^{n-k}$.
\end{proof}

\begin{thm}\label{thCap}
We have
\begin{eqnarray}
\Gamma=\Q,\qquad n\leq 5;\qquad \Gamma\neq \Q,\qquad n=6.\label{yyy6}
\end{eqnarray}
\end{thm}
\begin{proof}
In the cases $n\leq 4$, we have
\begin{eqnarray}
\Q&=&\Gamma^{\overline{1}}=\{T\in\Cl^\times:\quad T\Cl^{\overline{1}}T^{-1}\subseteq \Cl^{\overline{1}}\}\nonumber\\
&=&\{T\in\Cl^\times:\quad T\Cl^{1}T^{-1}\subseteq \Cl^{1}\}=\Gamma^1\nonumber
\end{eqnarray}
because $\Cl^{\overline{1}}=\Cl^1$. In the case $n=5$, we have for $T\in\Q$
$$T\Cl^1T^{-1}\subseteq \Cl^{\overline{1}}=\Cl^1 \oplus\Cl^5.$$
Suppose that
$$TvT^{-1}=w+\lambda e_{1\ldots 5},\qquad v, w\in\Cl^1$$
with some constant $\lambda$. Then
\begin{eqnarray}
\lambda e&=&e_{1\ldots 5}^{-1}TvT^{-1}-e_{1\ldots 5}^{-1}w=\la e_{1\ldots 5}^{-1}TvT^{-1}-e_{1\ldots 5}^{-1}w\ra_0\nonumber\\
&=&\la e_{1\ldots 5}^{-1}TvT^{-1} \ra_0= \la e_{1\ldots 5} v \ra_0=0,\nonumber
\end{eqnarray}
where we use the property $\la AB\ra_0=\la BA \ra_0$ of the scalar part operation \cite{Lounesto} (it is the projection onto the subspace $\Cl^0$)\footnote{This operation coincides with the trace of the corresponding matrix representation up to the scalar, which is the dimension of this representation (see \cite{Bulg}).}, and $e_{1\ldots 5}\in\Z$.

In the case $n=p=6$, $q=0$, consider the element
$$T=\frac{1}{\sqrt{2}}(e_{12}+e_{3456}) \in \Cl^{(0)}.$$
We have
$$\tilde{T} T=\frac{1}{2}(-e_{12}+e_{3456})(e_{12}+e_{3456})=e$$
and
$$T e_1 T^{-1}= \frac{1}{2}(e_{12}+e_{3456})e_1(-e_{12}+e_{3456})=-e_{23456}\notin \Cl^1,$$
but $e_{23456}\in\Cl^{\overline{1}}$. Thus $T\in\Q$, $T\notin\Gamma$, and $\Q\neq\Gamma$ in this case.
\end{proof}

Note that the following analogue of the statement of the previous theorem is known: in the cases $n\leq 5$, if $T\in \Cl^{(0)}$ and $\tilde{T}T=\pm e$, then $T\Cl^1T^{-1}\subseteq \Cl^1$.
That is why the spin group in the cases $n\leq 5$ (see, for example, \cite{Bulg}) can be defined as
$$\Spin(p,q)=\{T\in \Cl^{(0)},\, \tilde{T} T=\pm e\},\qquad n\leq 5.$$

\begin{lem}\label{lem7} We have
\begin{eqnarray}
&&\A=\Q=\P=\Z^\times(\Cl^{\times (0)}\cup\Cl^{\times (1)}),\qquad n\leq 3;\label{b1}\\
&&\Q\neq \A,\qquad n=4.\label{b2}
\end{eqnarray}
\end{lem}
\begin{proof} We have $\Q=\P$ for $n\leq 3$ by Lemma \ref{lemQ}. Let us prove that $\Q=\A$ for $n\leq 3$. If $T\in\A$, then $T\Cl^{\overline{01}}T^{-1}\subseteq\Cl^{\overline{01}}$, in particular,
$$T\Cl^1T^{-1}\subseteq \Cl^{0}\oplus\Cl^1.$$
Using the property $\la TUT^{-1} \ra_0=\la U \ra_0$  of the scalar part operation, we get
$$T\Cl^1T^{-1}\subseteq \Cl^1.$$
Thus, $\A\subseteq \Gamma$ in the cases $n\leq 3$. We have $\Gamma=\Q$ for $n\leq 5$, and $\Q\subseteq \A$. Thus we get $\A=\Q$ for $n\leq 3$.

In the case $n=4$, for the element
$$
S=e+2e_{123} \notin\Cl^{\times (0)}\cup\Cl^{\times (1)},
$$
we have
$$
\tilde{S}S=(e-2e_{123})(e+2e_{123})=e-4(e_{123})^2\in\Cl^{\times 0}.
$$
Thus $S\in\A$, $S\notin\Q$, and $\A\neq \Q$ in the case $n=4$.
\end{proof}

\begin{lem}\label{lemmagammakn} We have the following relations in the cases of small dimensions
\begin{eqnarray}
&&\Gamma^1=\Gamma^2=\Q=\P=\A,\qquad n=3;\\
&&\Gamma^1=\Gamma^3=\Q\neq\Gamma^2=\Q^\prime=\P,\qquad n=4;\\
&&\Gamma^1=\Gamma^2=\Gamma^3=\Gamma^4=\Q,\qquad n=5.
\end{eqnarray}
\end{lem}
\begin{proof} We have $\Gamma^1=\Q$ for $n\leq 5$. If we have $\Gamma^1\subseteq \Gamma^k\subseteq \Q$ for some fixed $k$ (see Lemma \ref{lemmagammakq}), then we can conclude that $\Gamma^1=\Gamma^k=\Q$ for this $k$. In the case $n=4$, we get $\Gamma^2=\Gamma^{\overline{2}}=\Q^\prime=\P$ from Theorem \ref{th3}.
\end{proof}

Note that $\Gamma^1 \neq \Gamma^3$ in the case $n=6$. This fact follows from Theorem \ref{th3} (we have $\Gamma^3=\Gamma^{\overline{3}}=\Q$ in the case $n=6$) and Theorem \ref{thCap} (we have $\Gamma\neq \Q$ in the case $n=6$). For example, the element $T=e+2e_{123456}$ is invertible because $$(e+2e_{123456})(e-2e_{123456})=e-4(e_{123456})^2,$$
is a nonzero scalar. We have
$$(e+2e_{123456})U(e-2e_{123456})=(e+2e_{123456})^2 U\in\Cl^3$$
for an arbitrary $U\in\Cl^3$ because $e_{123456}$ anticommutes with all odd elements. Also we have
$$(e+2e_{123456})e_{1}(e-2e_{123456})=(e+2e_{123456})^2 e_{1}\in\Cl^1\oplus\Cl^5.$$
Thus $T\notin\Gamma$ and $T\in\Gamma^3$.

\section{The corresponding Lie algebras}\label{sectLieAlg}

It is well-known (see, for example, \cite{BT, Bulg}) that the Lie algebra of the Lie group $\Gamma$ is the following subspace with respect to the operation of commutator
$$\mathfrak{\gamma}:=\left\lbrace
\begin{array}{ll}
\Cl^0\oplus\Cl^{2}, & \mbox{if $n$ is even,}\\
\Cl^0\oplus\Cl^2\oplus\Cl^{n}, & \mbox{if $n$ is odd,}
\end{array}\nonumber
\right.$$
with dimension
$$\dim \Gamma=\dim \mathfrak{\gamma}=\left\lbrace
\begin{array}{ll}
\frac{n(n-1)}{2}+1, & \mbox{if $n$ is even,}\\
\frac{n(n-1)}{2}+2, & \mbox{if $n$ is odd.}
\end{array}\nonumber
\right.$$

\begin{thm} The Lie groups $\P$, $\A$, $\B$, $\Q$, and $\Q^\prime$
have the corresponding Lie algebras\footnote{We present Lie algebras for all Lie groups considered in this paper for the sake of completeness. Some of them are known.} $\mathfrak{p}$, $\mathfrak{a}$, $\mathfrak{b}$, $\mathfrak{q}$, and $\mathfrak{q}^\prime$ illustrated in Table~\ref{table1} with corresponding dimensions.
Also we have
\begin{eqnarray}
&&\mathfrak{\gamma}\subseteq \mathfrak{q}\subseteq \mathfrak{p},\qquad \mathfrak{q}\subseteq \mathfrak{a},\qquad \mathfrak{q}\subseteq \mathfrak{b},\qquad\mathfrak{q}=\mathfrak{a}\cap \mathfrak{p}=\mathfrak{b}\cap \mathfrak{p}=\mathfrak{a}\cap \mathfrak{b};\nonumber\\
&&\mathfrak{q}=\mathfrak{p}=\mathfrak{a},\qquad n\leq 3;\label{liealg}\\
&&\mathfrak{q}\neq \mathfrak{p},\qquad \mathfrak{q}\neq \mathfrak{a},\qquad \mathfrak{p}\neq \mathfrak{a},\qquad n=4;\nonumber\\
&&\mathfrak{\gamma}=\mathfrak{q},\qquad n\leq 5;\qquad \mathfrak{\gamma}\neq \mathfrak{q},\qquad n=6.\nonumber
\end{eqnarray}
In the cases $n=0\mod 4$, we have
\begin{eqnarray}
&&\mathfrak{q}\subseteq \mathfrak{q}^\prime\subseteq \mathfrak{p},\qquad \mathfrak{\gamma}\subseteq \mathfrak{q}^\prime;\qquad \mathfrak{q}\neq \mathfrak{q}^\prime,\qquad n\geq 4;\label{liealg2}\\
&&\mathfrak{q}^\prime=\mathfrak{p},\qquad n=4;\qquad \mathfrak{q}^\prime\neq \mathfrak{p},\qquad n\geq 8.\nonumber
\end{eqnarray}
\end{thm}
Note that the statements (\ref{liealg}) are analogues of the corresponding statements (\ref{yyy4}), (\ref{rty}), (\ref{rty2}), (\ref{yyy7}), (\ref{yyy8}), (\ref{yyy6}), (\ref{b1}), (\ref{b2}) for the Lie groups . The statements (\ref{liealg2}) are analogues of the corresponding statements (\ref{QQprime}), (\ref{p1}), (\ref{p2}), (\ref{p3}), (\ref{yyy9}) for the Lie groups.

\begin{proof} To prove these statements, we use the standard facts about the relation between an arbitrary Lie group and the corresponding Lie algebra.

%
%
%
%
%

To calculate dimensions, we use (see, for example, \cite{Lie3})\footnote{As one of the anonymous reviewers noted, the dimensions of the four subspaces $\Cl^{\overline{k}}$, $k=0, 1, 2, 3$ in the case of more general algebras (GCSAsWI) are also known from other works.}
\begin{eqnarray}
\dim \Cl^{\overline{2}}&=&2^{n-2}-2^{\frac{n-2}{2}}\cos (\frac{\pi n}{4}),\qquad \dim\Cl^{(0)}=2^{n-1},\nonumber\\
\dim \Cl^{\overline{23}}&=&\dim \Cl^{\overline{2}}+\dim \Cl^{\overline{3}}=2^{n-2}-2^{\frac{n-2}{2}}\cos (\frac{\pi n}{4})\nonumber\\
&+&2^{n-2}-2^{\frac{n-2}{2}}\sin (\frac{\pi n}{4})=2^{n-1}-2^{\frac{n-1}{2}}\sin (\frac{\pi(n+1)}{4}),\nonumber\\
\dim \Cl^{\overline{12}}&=&\dim \Cl^{\overline{2}}+\dim \Cl^{\overline{1}}=2^{n-2}-2^{\frac{n-2}{2}}\cos (\frac{\pi n}{4})\nonumber\\
&+&2^{n-2}+2^{\frac{n-2}{2}}\sin (\frac{\pi n}{4})=2^{n-1}-2^{\frac{n-1}{2}}\cos (\frac{\pi(n+1)}{4}).\nonumber
\end{eqnarray}
All statements (\ref{liealg}), (\ref{liealg2}) are easily verified using definitions of the corresponding Lie algebras.
\end{proof}

\section{The cases of small dimensions $n\leq 5$}\label{sectSmallDim}

Let us write down all the Lie groups considered in this paper and the corresponding Lie algebras for the cases of small dimensions $n\leq 5$.

If $n=1$, then the Clifford algebra $\Cl$ is a commutative algebra and all considered Lie groups coincide with $\Cl^\times$. The corresponding Lie algebra is $\Cl$ considered w.r.t. the operation of commutator.

If $n=2$, then we have two different groups
\begin{eqnarray}
&&\Gamma^{\overline{0}}=\Gamma^{0}= \Gamma^{\overline{03}}= \Gamma^{\overline{12}}=\B=\Cl^\times,\nonumber\\
&&\Gamma^1=\Gamma^2=\Gamma^{\overline{2}}=\Gamma^{\overline{23}}= \Gamma^{\overline{01}}= \Gamma^{(0)}= \Gamma^{(1)}=\Gamma^{\overline{1}}\nonumber\\
&&=\Gamma=\Q=\P=\A=\Cl^{\times (0)}\cup\Cl^{\times (1)},\nonumber
\end{eqnarray}
with the corresponding Lie algebras $\Cl$ and $\Cl^{02}$ respectively.

If $n=3$, then we have two different groups
\begin{eqnarray}
&&\Gamma^{\overline{0}}=\Gamma^{0}=\Gamma^3=\Gamma^{\overline{3}}= \Gamma^{\overline{03}}= \Gamma^{\overline{12}}=\B=\Cl^\times,\nonumber\\
&&\Gamma^{1}=\Gamma^{(0)}=\Gamma^{(1)}=\Gamma^{\overline{1}}= \Gamma^2=\Gamma^{\overline{2}}=\Gamma^{\overline{23}}=\Gamma^{\overline{01}}\nonumber\\
&&=\Gamma=\Q=\P=\A=\Z^\times\Cl^{\times (0)},\nonumber
\end{eqnarray}
with the corresponding Lie algebras
$\Cl$ and $\Cl^{023}$ respectively.

If $n=4$, then we have five different groups
\begin{eqnarray*}
&&\Gamma^{0}=\Cl^\times,\nonumber\\
&&\Gamma^{(0)}=\Gamma^{(1)}=\Gamma^{\overline{0}}=\Gamma^4=\Gamma^{\overline{2}}= \Gamma^2=\P=\Q^\prime=\Cl^{\times (0)}\cup\Cl^{\times (1)},\nonumber\\
&&\Gamma^{1}=\Gamma^{\overline{1}}= \Gamma^3=\Gamma^{\overline{3}}=\Gamma=\Q=\{T\in\Cl^{\times (0)}\cup\Cl^{\times (1)}:\,\, \tilde{T} T\in \Cl^{0 \times}\},\nonumber\\
&&\Gamma^{\overline{03}}=\Gamma^{\overline{12}}=\B=\{T\in\Cl^\times:\quad \hat{\tilde{T}} T\in \Cl^{\times 0}\},\nonumber\\
&&\Gamma^{\overline{23}}=\Gamma^{\overline{01}}=\A=\{T\in\Cl^\times:\quad \tilde{T} T\in \Cl^{\times 0}\}.\nonumber
\end{eqnarray*}
with the corresponding Lie algebras
$\Cl$, $\Cl^{024}$, $\Cl^{02}$, $\Cl^{012}$, and $\Cl^{023}$ respectively.

If $n=5$, then we have five different groups
\begin{eqnarray}
&&\Gamma^{0}=\Gamma^5=\Cl^\times,\nonumber\\
&&\Gamma^{(0)}=\Gamma^{(1)}=\P=\Z^\times \Cl^{\times (0)},\nonumber\\
&&\Gamma^{1}=\Gamma^{2}=\Gamma^{3}=\Gamma^{4}=\Gamma^{\overline{1}}= \Gamma^{\overline{2}}=\Gamma^{\overline{3}}=\Gamma^{\overline{0}}\nonumber\\
&&=\Gamma=\Q=\{T\in\Z^\times \Cl^{\times (0)}:\quad \tilde{T} T\in \Cl^{0 \times}\},\nonumber\\
&&\Gamma^{\overline{03}}=\Gamma^{\overline{12}}=\B=\{T\in\Cl^\times:\quad \hat{\tilde{T}} T\in \Cl^{\times 0 n}\},\nonumber\\
&&\Gamma^{\overline{23}}=\Gamma^{\overline{01}}=\A=\{T\in\Cl^\times:\quad \tilde{T} T\in \Cl^{\times 0}\}.\nonumber
\end{eqnarray}
with the corresponding Lie algebras
$\Cl$, $\Cl^{0245}$, $\Cl^{025}$, $\Cl^{012}$, and $\Cl^{0235}$  respectively.

The relations between all the considered Lie groups are illustrated in Figure \ref{figure1}.

\section{Conclusions}\label{sectConc}

In this paper, we present an explicit form of the groups of elements that define inner automorphisms preserving different naturally defined subspaces (subspaces defined by the reversion and the grade involution and subspaces of fixed grades) of the real or complex Clifford algebra $\Cl$
\begin{eqnarray}
&&\Gamma^{(0)}=\Gamma^{(1)}=\P,\qquad \Gamma^{\overline{01}}=\Gamma^{\overline{23}}=\A,\qquad \Gamma^{\overline{03}}=\Gamma^{\overline{12}}=\B,\nonumber\\ &&\Gamma^{\overline{1}}=\Gamma^{\overline{3}}=\Q,\qquad \Gamma^{\overline{0}}=\Gamma^{\overline{2}}=\Q \quad\mbox{or}\quad \Q^\prime,\qquad \Gamma^k, \quad k=0, 1, \ldots, n.\nonumber
\end{eqnarray}
We present the corresponding Lie algebras. The relations between Lie groups (and Lie algebras) are illustrated in Table~\ref{table1} and Figure~\ref{figure1}.
One of this groups (the Clifford group $\Gamma$, which preserves the subspace of grade~$1$) is well-known and is widely used in the theory of spin groups. The other groups $\A, \B, \Q, \Q^\prime$ are related to the norm functions $\psi(T)=\tilde{T}T$ and $\chi(T)=\hat{\tilde{T}}T$ and can be also used in different applications.

Note that the spin groups are defined as normalized subgroups (using the norm functions $\psi$ and $\chi$) of the Lipschitz group
\begin{eqnarray}
\Gamma^\pm=\{T\in\Cl^{\times (0)}\cup\Cl^{\times (1)}:\,\, T \Cl^{1}T^{-1}\subseteq \Cl^{1}\},\nonumber
\end{eqnarray}
which is a subgroup of the Clifford group $\Gamma$ (we have $\Gamma^\pm\subseteq\Gamma$ and $\Gamma=\Z^\times\Gamma^\pm$). We can consider normalized subgroups (using the same norm functions $\psi$ and $\chi$) of the groups $\P$, $\A$, $\B$, $\Q$, $\Q^\prime$. These normalized subgroups can be interpreted as generalizations of the spin groups and can be used in different applications of Clifford algebras. In the papers \cite{Lie1, Lie2, Lie3}, we have already considered some of these groups (see the groups $\G^{23}_{p,q}$, $\G^{12}_{p,q}$, $\G^2_{p,q}$ in \cite{Lie3}) and found isomorphisms between these groups and classical matrix Lie groups. Some of these groups are related to automorphism
groups of the scalar products on the spinor spaces (see \cite{Lounesto, Port, BT, Abl}). 

All the groups considered in this paper contain spin groups as subgroups. The real Clifford algebras are isomorphic to the matrix algebras over $\R$, $\R\oplus\R$, $\C$, ${\mathbb H}$, or ${\mathbb H}\oplus{\mathbb H}$ depending on $p-q\mod 8$ and the complex Clifford algebras are isomorphic to the matrix algebras over $\C$ or $\C\oplus\C$ depending on $n\mod 2$. In the opinion of the author, the structure of naturally defined fundamental subspaces (subspaces of fixed grades and subspaces defined by the reversion and the grade involution) favourably distinguishes Clifford algebras from the corresponding matrix algebras, when we use them for applications.  Groups that preserve different subspaces of Clifford algebra under similarity transformation may be of interest for different applications of Clifford algebras  -- in physics, engineering, robotics, computer vision, image and signal processing (see about applications of various groups in physics using Clifford algebra formalism, for example, in \cite{Marchuk, Snygg}).

\subsection*{Acknowledgment}

The study presented in this paper was stimulated by discussions with Prof. A.~Odzijewicz and other participants during the International Conference on Mathematical Methods in Physics (Morocco, Marrakesh, 2019). The author is grateful to the organizers and the participants of this conference.

The author is grateful to the anonymous reviewers (especially to the second reviewer, who pointed out the feasibility of some of the results of this paper not only for the Clifford algebras but for more general algebras, as well as for other important comments) for their careful reading of the paper and helpful
comments on how to improve the presentation.

This work is supported by the grant of the President of the Russian Federation (project MK-404.2020.1).


\begin{thebibliography}{1}
\bibitem{Abl} R.~Ab\l amowicz, B.~Fauser, On the transposition anti-involution in real Clifford algebras III: the automorphism group of the transposition scalar product on spinor spaces, \textit{Linear and Multilinear
Algebra}, 2012; 60(6): 621-644.
\bibitem{ABS} M.~Atiyah, R.~Bott, A.~Shapiro, \textit{Clifford Modules}, Topology \textbf{3} (1964) 3-38.
 \bibitem{BT} I.~M.~Benn and R.~W.~Tucker, \textit{An introduction to Spinors and Geometry with Applications in Physics} (Bristol, 1987).
 \bibitem{Chev} C.~Chevalley, \textit{The algebraic theory of Spinors and Clifford algebras}, Springer, (1996).
 \bibitem{Diedonne} J.~Dieudonne, \textit{La geometrie des groupes
classiques}, Springer-Verlag, (1971).
\bibitem{Lasenby} C. Doran, A. Lasenby, \textit{Geometric Algebra for Physicists}, Cambridge Univ. Press, Cambridge 2003.
\bibitem{Helm} J.~Helmstetter, \textit{The Group of Classes of Involutions of Graded Central Simple Algebras.} In: Abłamowicz R. (eds) Clifford Algebras. Progress in Mathematical Physics, vol 34. Birkhäuser Boston. 2004.
\bibitem{Hestenes} D. Hestenes, G. Sobczyk, \textit{Clifford Algebra to Geometric Calculus}, Reidel Publishing Company, Dordrecht Holland, 1984..
\bibitem{LM} H.~B.~Lawson, M.-L.~Michelsohn, \textit{Spin geometry} (Princeton, Princeton Univ. Press, 1989).
\bibitem{Lounesto} P.~Lounesto, \textit{Clifford Algebras and Spinors} (Cambridge, Cambridge Univ. Press, 1997).
\bibitem{Marchuk} N.~Marchuk, \textit{Field theory equations}, Amazon, ISBN 9781479328079, 290 p., (2012).
 \bibitem{Port} I.~R.~Porteous, \textit{Clifford Algebras and the Classical Groups} (Cambridge University Press, 1995).
 \bibitem{Bulg} D.~S.~Shirokov, \textit{Clifford algebras and their applications to Lie groups and spinors}, Proceedings of the 19 International Conference on Geometry, Integrability and Quantization, Avangard Prima, Sofia, 2018, 11-53, arXiv:1709.06608.
   \bibitem{Paulispin} D.~S.~Shirokov, The use of the generalized Pauli's theorem for odd elements of Clifford algebra to analyze relations between spin and orthogonal groups of arbitrary dimensions, \textit{Vestn. Samar. Gos. Tekhn. Univ. Ser. Fiz.-Mat. Nauki}, 1(30), 2013, 279 - 287.
 \bibitem{DAN} D.~S.~Shirokov, Extension of Pauli's theorem to Clifford algebras, \textit{Dokl. Math.}, \textbf{84}:2, 699-701 (2011).
 \bibitem{quat} D.~S.~Shirokov, Classification of elements of Clifford algebras according to quaternionic types, \textit{Dokl. Math.}, \textbf{80}:1, 610-612 (2009).
 \bibitem{quat1} D.~S.~Shirokov, Quaternion typification of Clifford algebra elements, \textit{Adv. Appl. Clifford Algebr.}, \textbf{22}:1, 243-256 (2012).
 \bibitem{quat2} D.~S.~Shirokov, Development of the method of quaternion typification of Clifford algebra elements, \textit{Advances in Applied Clifford Algebras}, \textbf{22}:2, 483-497 (2012).
\bibitem{Spin} D.~S.~Shirokov, Calculations of elements of spin groups using generalized Pauli's theorem, \textit{Advances in Applied Clifford Algebras}, \textbf{25}:1, 227-244 (2015).
\bibitem{Lie1} D.~S.~Shirokov, Symplectic, Orthogonal and Linear Lie Groups in Clifford Algebra, \textit{Advances in Applied Clifford Algebras}, \textbf{25}:3, 707-718 (2015).
\bibitem{Lie2} D.~S.~Shirokov, On Some Lie Groups Containing Spin Group in Clifford Algebra, \textit{Journal of Geometry and Symmetry in Physics}, \textbf{42}, 73-94 (2016).
\bibitem{Lie3} D.~S.~Shirokov, Classification of Lie algebras of specific type in complexified Clifford algebras, \textit{Linear and Multilinear Algebra}, \textbf{66}:9, 1870-1887 (2018).
\bibitem{genR} D.~S.~Shirokov, Method of generalized Reynolds operators in Clifford algebras, arXiv:1409.8163
\bibitem{Snygg} J.~Snygg, \textit{Clifford Algebra - A Computational Tool For Physicists}, Oxford University Press, New York 1997.
\bibitem{Wall} C.~T.~C.~Wall, \textit{Graded Brauer groups}, J. Reine Angew. Math., \textbf{213}, 187–199 (1963/1964).
\bibitem{Wall2} C.~T.~C.~Wall, \textit{Graded algebras, anti-involutions, simple groups and symmetric spaces}, Bull. Amer. Math. Soc. \textbf{74}, 198–202 (1968).
\end{thebibliography}
\end{document}